%% file: pseudoexp_excellence.tex
\documentclass[11pt,a4paper,oneside]{amsart}
\usepackage{amssymb,amsmath,amsthm,hyperref,enumerate}
\usepackage[mathscr]{eucal}
\usepackage[textwidth=14cm,textheight=23cm]{geometry}

\input{macros}

\newtheorem{prop}{Proposition}
\newtheorem{cor}[prop]{Corollary}
\newtheorem{theorem}[prop]{Theorem}
\newtheorem{lemma}[prop]{Lemma}

\theoremstyle{definition}

\newtheorem{notation}[prop]{Notation}

\newtheorem{remark}[prop]{Remark}

\newcommand{\G}{\mathbb{G}}

\newcommand{\EC}{\ensuremath{\mathcal{EC}^*_{\mathrm{st,ccp}}}}
\newcommand{\qmec}{quasiminimal excellent class}

\newcommand{\quot}[2]{\frac{#1}{#2}}

\newcommand{\alphabar}{{\bar{\alpha}}}
\newcommand{\betabar}{{\bar{\beta}}}
\newcommand{\gammabar}{{\bar{\gamma}}}

\title[Categoricity of Zilber's exponential fields]{Excellence and uncountable categoricity of Zilber's exponential fields}
\author{Martin Bays and Jonathan Kirby}
\date{version 1.0, \today}

\begin{document}

\begin{abstract}
We prove that Zilber's class of exponential fields is quasiminimal excellent and hence uncountably categorical, filling two gaps in Zilber's original proof.
\end{abstract}

\maketitle

\section{Introduction}

In  \cite{Zilber05peACF0}, Zilber defined his exponential fields to be models in a certain class \EC. The purpose of this note is to give a proof of the following categoricity theorem, filling gaps in the proof from that paper.
\begin{theorem}\label{categoricity}
 For each cardinal $\lambda$, there is exactly one model $B_\lambda$ in \EC\ of exponential transcendence degree $\lambda$, up to isomorphism. Furthermore, the cardinality of $B_\lambda$ is $\lambda + \aleph_0$, so \EC\ is categorical in all uncountable cardinals.
\end{theorem}
In particular there is a unique model, $\mathbb{B}$, of cardinality $2^{\aleph_0}$ and so it makes sense to ask, as Zilber does, whether the complex exponential field $\Cexp$ is isomorphic to $\mathbb{B}$.

In another paper \cite{BHHKK} we prove a result about quasiminimal excellent classes which gives a way to avoid one of the holes in Zilber's proof: instead of proving the excellence axiom for this class we show it is redundant by proving it in general. However, we believe that the direct proof given here is of independent interest.

In section~2 of this paper we outline the definition of Zilber's exponential fields, and the concepts of strong extensions and partial exponential subfields. We then give the first-order language in which we work. Section~3 introduces quasiminimal excellent classes and, largely by reference to Zilber's work, shows that \EC\ satisfies all of the axioms except the excellence axiom. We also give some corollaries of these axioms which are used later. Finally, section~4 contains the proof of the excellence property and the categoricity theorem, along with the relevant definitions.

\section{Zilber's exponential fields}

\subsection{The axioms}

The full definition of the class \EC\ would take too much space to give, so below we give the list of axioms without full details of the last two. The reader is referred to Zilber's paper, to \cite{Marker06}, to \cite{FPEF} or to \cite{NAPE} where a discussion of the axioms is given. A structure $\tuple{F; +,\cdot, \exp}$ is in \EC\ if it satisfies the following five axioms.
\begin{description}
 \item[1. ELA-field] $F$ is an algebraically closed field of characteristic zero, and its exponential map $\exp$ is a surjective homomorphism from its additive group to its multiplicative group.

 \item[2. Standard kernel] the kernel of the exponential map is an infinite cyclic group generated by a transcendental element $\tau$.

 \item[3. Schanuel Property] The \emph{predimension function} 
\[\delta(\xbar) \leteq \td(\xbar, \exp(\xbar))- \ldim_\Q(\xbar)\]
satisfies $\delta(\xbar) \ge 0$ for all tuples $\xbar$ from $F$.

\item[4. Strong exponential-algebraic closedness] This axiom roughly states that every system of exponential polynomial equations with coefficients from $F$ has a solution in $F$, unless doing so would contradict the Schanuel property.

\item[5. Countable Closure Property] An exponential field $F$ has an associated pregeometry, \emph{exponential algebraic closure}, written $\ecl^F$. The dimension notion for this pregeometry is called \emph{exponential transcendence degree}. The last axiom states that for each finite subset $C$ of $F$, the exponential algebraic closure $\ecl^F(C)$ of $C$ in $F$ is countable. 
\end{description}

The paper \cite{FPEF} gives an explicit construction (Construction~6.4) of a countable model $B_\lambda$ of exponential transcendence degree $\lambda$, for each countable cardinal $\lambda$. In Corollary~6.10 of the same paper it is proved that they are the only countable models in \EC\ up to isomorphism. In particular there is a unique countable model $B_{\aleph_0}$ of exponential transcendence degree $\aleph_0$, which we will henceforth write just as $B$.

\subsection{Strong extensions}
If $A$ is a subset of $B$, we write $\gen{A}$ for the $\Q$-vector subspace of $B$ spanned by $A$. If $A_1 \subs A_2 \subs B$ we say that $A_1$ is \emph{strong} in $A_2$ and write $A_1 \strong A_2$ if for all finite tuples $\bbar \in A_2$ we have 
\[\delta(\bbar/A_1) \leteq \td(\bbar,\exp(\bbar),A_1,\exp(A_1)/A_1,\exp(A_1)) - \ldim_\Q(\bbar/A_1) \ge 0.\]

It follows from axiom 3, the Schanuel property, that given any subset $A$ of $B$, there is a smallest $\Q$-linear subspace $\hull{A}$ of $B$ which contains $A$ and is strong in $B$. Furthermore if $\abar$ is finite then $\hull{\abar}$ is a finite-dimensional $\Q$-vector space, and moreover if $A\strong B$ and $\abar$ is a finite tuple, then $\ldim_\Q(\hull{A\abar}/A)$ is finite. 

We can define the closure operator $\ecl$ by 
\[ \ecl(A) := \bigcup \{ \abar \in B \;|\; \delta(\abar / \hull{A}) = 0 \} .\]

\subsection{Partial exponential subfields}
We outline the necessary algebra of exponential fields. See section~3 of \cite{FPEF} for full details. Given $A \subs B$, the partial exponential subfield $F_A$ of $B$ generated by $A$ is the subfield of $B$ generated by $A \cup \exp(A)$, equipped with the restriction to $A$ of the exponential map of $B$.

Suppose $A \subs B$ and $\abar$ is a finite tuple. Then the isomorphism type of $F_{A\cup\abar}$ as an extension of $F_A$ is given by the sequence of varieties $(V_m)_{m \in \N^+}$ where $V_m \leteq \loc(\quot{\abar}{m},e^{\quot{\abar}{m}}/F_A)$ is the smallest subvariety of $\G_a^n\times\G_m^n$ defined over $F_A$ and containing $(\quot{\abar}{m},e^{\quot{\abar}{m}})$. 

The following proposition is Theorem~3 of \cite{BZ11}.
\begin{prop}\label{prop:thumbtack}
Suppose $F_A$ is a partial exponential subfield of $B$, and there is an algebraically independent subset $Y$ of $B$ and subsets $Y_1,\ldots,Y_N$ of $Y$ such that $F_A$ is the subfield of $B$ generated by $\bigcup_{i=1}^N \acl(Y_i)$. Then for any finite tuple $\abar$ from $B$, there is $m_0 \in \N^+$ such that the sequence of varieties $(V_m)_{m\in\N^+}$ as given above is determined by the single variety $V_{m_0}$.
\end{prop}
We say that the variety $V_{m_0}$ is \emph{Kummer-generic} over $F_A$ in this case.

\subsection{A different language}

The proof of categoricity goes via showing that the class \EC\ satisfies the axioms of a \qmec. These axioms are rather sensitive to the choice of language because they imply a form of quantifier elimination, so we have to add relations to the language in order to get this quantifier elimination. On the other hand, in order to prove the axioms we need the substructure generated by a finite set to be rather small, and in particular we need the exponential map to be given by a relation symbol rather than a function symbol. The axioms of \qmec es we use refer only to the countable models of \EC\ and these all embed in $B$, so it is enough to consider all our structures to be substructures of $B$.

The language we use is the language $\tuple{+,0, (\lambda\cdot)_{\lambda \in \Q}}$ of $\Q$-vector spaces, expanded by $r$-ary relations $R_{V,n}$ for each $n,r \in \N$ and each algebraic subvariety $V$ of $\ga^{r+n} \cross \gm^{r+n}$, defined and irreducible over $\Q$, where $R_{V,n}$ names the formula given by 
\[(\exists \ybar)(\forall \mbar \in \Q^{r+n}) \left[ (\xbar,\ybar,e^\xbar,e^\ybar) \in V \wedge \left[ \sum_{i=1}^r m_{i}x_i + \sum_{i=1}^n m_{i+r}y_i = 0 \to \bigwedge_{i=r+1}^{r+n} m_i = 0\right]\right].
\]

We call this language $L$. Note that by taking $r=3$, $n=0$ and $V$ to be given by $x_1x_2=x_3$ we get the graph of multiplication, and taking $r=2$, $n=0$ and $V$ to be given by $x_2 = x'_1$ we get the graph of exponentiation. Also, the subset of rationals $\Q$ is definable in the original language $\tuple{+,\cdot,\exp}$, since the integers can be defined as the additive stabiliser of the kernel of the exponential map. So $B$ regarded as an $L$-structure has the same first-order definable sets as $B$ regarded as a structure in the original language $\tuple{+,\cdot,\exp}$. Zilber uses a language $L^*$ which is very similar to our $L$, but which is not explicitly an expansion by first-order definable relations, only by \Loo-definable relations.  Note that the $L$-substructure of $B$ generated by a subset $A$ is just its \Q-linear span $\gen{A}$. Our language $L$ is closely related to strong partial exponential subfields as the next lemma shows.
\begin{lemma}\label{pE-subfields}
Suppose that $A \strong B$ and $\abar$, $\bbar$ are finite tuples such that $A\cup\abar \strong B$ and $A \cup \bbar \strong B$. Suppose also that there is an isomorphism $F_{A \cup\abar} \iso F_{A \cup \bbar}$ sending $\abar$ to $\bbar$ and fixing $F_A$ pointwise. Then $\qftp_L(\abar/A) = \qftp_L(\bbar/A)$.
\end{lemma}
\begin{proof}
This is \cite[Lemma~5.7]{Zilber05peACF0}, in our language.
\end{proof}

\section{Quasiminimal excellent classes}
Zilber's method of proof for Theorem~\ref{categoricity} was to show that \EC\ is a \qmec, that is, a class of structures equipped with pregeometries satisfying certain axioms. We follow the presentation of \qmec es from \cite{OQMEC}, which has axioms 0, I, II, III and IV.
\begin{prop}
The class \EC\ satisfies axioms 0, I, and IV of \qmec es.
\end{prop}
\begin{proof}
Axiom 0 (closure under isomorphism) is immediate. Axiom I.1 states that the operation $\ecl$ really is a pregeometry which satisfies the countable closure property. The countable closure property is part of the definition of \EC, and the axioms of pregeometries are given by \cite[Proposition~4.11]{Zilber05peACF0}. For an alternative approach, see \cite{EAEF}. Axiom I.2 states that closed sets with respect to the pregeometry are models in the class, which is \cite[Lemma~5.3]{Zilber05peACF0}. Axiom I.3 states that the pregeometry is determined by quantifier-free $L$-formulas. The language $L$ is chosen to make this true and it is not hard to see from the definition of $\ecl$ given above. Alternatively, see \cite[Lemma~5.6]{Zilber05peACF0} for Zilber's account for his language $L^*$; the same proof works for our language $L$. 

Axiom IV.1 (closure under unions of chains) is immediate, and since we have the model $B$, axiom IV.2 (the existence of an infinite-dimensional model) follows.
\end{proof}

For Axioms~II and III, the proofs in \cite{Zilber05peACF0} are incomplete.
Axiom~III we handle in the next section; we consider Axiom~II here. The
statement of \cite[Proposition~5.9(i)]{Zilber05peACF0} is technically false,
since it makes no assumption that the submodel $G$ be closed. This assumption is not part of Zilber's version of Axiom~II, but without it the axiom is false for \EC. Moreover, the proof contains a gap where it assumes that the realisation of a type over
$C_2$ is generic over $GC_2$. If $G$ is of infinite dimension, the existence
of such a generic realisation is not immediate. We fill this gap in the
following Proposition.

\begin{prop}
  The class \EC\ satisfies axiom II of \qmec es.
\end{prop}
\begin{proof}
  We write $a \equiv_C b$ to mean $\qftp(a/C)=\qftp(b/C)$.
  Part (i) of axiom II, stating uniqueness of generic qf-types, follows
  from the definition of $\ecl$ and Lemma~\ref{pE-subfields}.

  For part (ii), suppose $G$ is a closed submodel of $B$ or $G=\emptyset$, and
  $c$ and $d$ are tuples in $B$, $d \in \ecl(Gc)$, and $c' \equiv_{G} c$; we
  want to find $d'$ such that $c'd' \equiv_{G} cd$.
  If $G=\emptyset$, this follows from the definition of $L$ and the strong
  exponential-algebraic closedness axiom. We will use this result
  in the course of the remainder of the proof of this proposition, and will
  refer to it as \emph{$\aleph_0$-homogeneity over $\emptyset$}.

  
  There remains the case that $G$ is a closed submodel.
  Extending $d$, we may assume that $c\in \gen{Gd}$.
  Since $\hull{Gd}$ is finitely generated over $G$, we may further assume
  that $\gen{Gd} \strong B$.
  By the $N=1$ case of Proposition~\ref{prop:thumbtack}, we may further assume
  that $V := \loc(d,\exp(d)/G)$ is Kummer-generic over $F_G$.
  Say $G_0 \strong G$ is finitely generated and such that $V$ is over $G_0$
  and $c\in \gen{G_0d}$. Note then that $d\in\ecl(G_0c)$.
  By $\aleph_0$-homogeneity over $\emptyset$, say $c'd' \equiv_{G_0} cd$.
  $G$ is $\ecl$-independent over $G_0$ from $d$, hence from $c$, hence from $c'$,
  hence from $d'$.
  Say $X$ is an $\ecl$-basis for $G$ over $G_0$. Then by uniqueness of
  generics (axiom II(i)), $d' \equiv_{G_0X} d$.
  Now $(d,\exp(d))$ is ACF-independent from $G=\ecl(G_0X)$ over $G_0$,
  so by $\aleph_0$-homogeneity over $\emptyset$, the same holds of
  $(d',\exp(d'))$. 
  So $(d',\exp(d'))$ is generic over $G$ in $V$, so by $\gen{Gd}\strong B$,
  Kummer-genericity of $V$, and Lemma~\ref{pE-subfields}, we have
  $d' \equiv_{G} d$ as required.
\end{proof}

We will make use of two corollaries of these axioms.
\begin{cor}\label{QM corollary} Let $H$ be $\ecl$-closed in $B$ and let $X,Y \subs B$ each be $\ecl$-independent over $H$, and of the same cardinality. Then there is an automorphism of $B$ fixing $H$ pointwise, which sends $X$ to $Y$.
\end{cor}
\begin{proof}
We apply Theorem~2.1 of \cite{OQMEC}. 
\end{proof}

\begin{cor}\label{language}
Let $\abar,\bbar$ be finite tuples from $B$. The following are equivalent.
 \begin{enumerate}[(i)]
  \item  $\qftp_L(\abar/\emptyset) = \qftp_L(\bbar/\emptyset)$ 
   \item  There is an automorphism of $B$ taking $\abar$ to $\bbar$.
  \item  For every $\Loo$-formula $\phi(\xbar)$, we have $B \models \phi(\abar)$ iff $B \models \phi(\bbar)$. 
 \end{enumerate}
\end{cor}
\begin{proof}
(i) $\implies$ (ii) follows from axiom II via a back-and-forth argument. (ii) $\implies$ (iii) and (iii) $\implies$ (i) are immediate.
\end{proof}


\section{Excellence}
Axiom III of \qmec es, the excellence property, is the most technical and it is where the most significant gap in Zilber's proof is. Specifically, in the proof of \cite[Proposition 5.15]{Zilber05peACF0} there is no reason why $A'B'$ should not lie in $C$, and then $V'$ would not contain $V_0$. Indeed, that proposition as stated is false, because the definition of finitary used there does not give sufficiently strong hypotheses.

The excellence property states that types over certain configurations of submodels, called \emph{crowns}, are determined over finite subsets. Let $X$ be an exponential transcendence base for $B$, let $X_1,\ldots,X_N \subs X$ with $\bigcup_{i=1}^N X_i = X$ and such that $X_0 \leteq \bigcap_{i=1}^N X_i$ is cofinite in $X$.  The crown associated to this data is $C = \gen{\bigcup_{i =1}^N \ecl^B(X_i)}$. We will prove axiom III of quasiminimal excellent classes in the special case of crowns $C$ of this form. 

\begin{prop}\label{excellence}
Let $r \in \N$ and $\abar \in B^r$. Then there is a finite tuple $c_0$ from $C$ such that if $\bbar\in B^r$ and $\qftp_L(\bbar/c_0) = \qftp_L(\abar/c_0)$ then $\qftp_L(\bbar/C) = \qftp_L(\abar/C)$.
\end{prop}
The full axiom III is the same statement but for more general crowns where we do not require that $\bigcup_{i=1}^N X_i = X$ or that $X_0$ is cofinite in $X$. In fact we have not been able to prove the full axiom III. Specifically, the proof of Lemma~\ref{field indep} below makes essential use of the fact that $X_0$ is infinite.

\begin{proof}[Proof of Theorem~\ref{categoricity} assuming Proposition~\ref{excellence}]
 We note that the proof of categoricity of \qmec es in \cite{OQMEC} (specifically Theorem~3.3) uses only crowns of exactly the form given, where the intersection $X_0$ of the $X_i$ is infinite, and each $X_i \minus X_0$ is finite. Hence it is enough to prove Proposition~\ref{excellence} in place of axiom III.
\end{proof}

\begin{notation}
Write $[N] = \{1,\ldots,N\}$, and for $s\subs [N]$ let $Z_s = \bigcap_{i\in [N]\minus s} X_i$ and let $A_s = \ecl^B(Z_s)$. Then $A_{[N]\minus \{i\}} = \ecl^B(X_i)$ and $C = \gen{\bigcup_{s \subsetneq [N]} A_s}$. For $s \subs [N]$ we write $A_{\ngeqslant s} \leteq 
\bigcup \class{A_t}{t \subs [N], t \not\sups s}$ and $A_{<s} \leteq \bigcup \class{A_r}{r \subsetneq s}$.

\end{notation}

\begin{lemma}\label{field indep}
  The system $(A_s)_{s \subsetneq [N]}$ is an independent system in the sense of Shelah with respect to ACF${}_0$-non-forking independence. That is, for each $s$, 
\[\findep{A_s}{A_{\ngeqslant s}}{A_{<s}}{ACF_0}\]
or equivalently, for each finite tuple $\abar \in A_s$, $\td(\abar/A_{\ngeqslant s}) = \td(\abar/A_{<s})$. 
\end{lemma}

\begin{proof}
Suppose the conclusion fails for $s$. So there are $a_0,\abar \in A_s$ and $\bbar_t \in A_t$ for each $t \subs [N]$ with $t \not\sups s$ such that, writing $\bbar$ for the tuple of all the $\bbar_t$, we have $a_0 \in \acl(\abar, \bbar) \minus \acl(\abar, A_{<s})$. Let $Z_0$ be a finite subset of $X_0$ such that $a_0,\abar \in \ecl^B(Z_0 \cup (Z_s \minus X_0))$ and $\bbar_t \in \ecl^B(Z_0 \cup (Z_t \minus X_0))$.

Now let $W_s = X \minus Z_s$, a finite subset of $X$, and choose a subset $W_0 \subs X_0 \minus Z_0$ of the same size. Then, by Corollary~\ref{QM corollary}, there is an automorphism $\sigma$ of $B$ fixing $\ecl^B(Z_0 \cup Z_s)$ pointwise, and sending $W_s$ to $W_0$. Then $\sigma(a_0) = a_0$ and $\sigma(\abar) = \abar$, but $\sigma(\bbar_t) \in \ecl^B(Z_t \cap Z_s) = A_r$ for some $r \subsetneq s$.

Then $a_0 \in \acl(\abar,\sigma(\bbar)) \subs \acl(\abar,A_{<s})$, a contradiction.
\end{proof}

\begin{lemma}\label{lem:thumbtack}
There is a transcendence base $Y$ of $C$ and subsets $Y_s$ of $Y$ for $s \subsetneq [N]$ such that $A_s = \acl( \bigcup_{t \subs s}Y_t)$. In particular, setting $Y^i = \bigcup_{s \subs [N] \minus \{i\}} Y_s$, we see that $C = \gen{\bigcup_{i =1}^N \acl(Y^i)}$. 
\end{lemma}
\begin{proof}
List the proper subsets $s$ of $[N]$ as $\emptyset = s_1, s_2, \ldots, s_{2^N-1}$ in such a way that $s_i \not\supseteq s_j$ for $i < j$. Inductively choose $Y_{s_i}$ a transcendence base for $A_{s_i}$ over $\bigcup_{r \subsetneq s_i} A_r$. By Lemma~\ref{field indep}, the set $Y = \bigcup_s Y_s$ is algebraically independent, hence a transcendence base for $C$.
\end{proof}

\begin{proof}[Proof of Proposition~\ref{excellence}]
By \cite[Lemma~5.14]{Zilber05peACF0}, $C \strong B$. Let $\alphabar \in B^n$ be a $\Q$-linear basis of $\hull{C \abar}$ over $C$ and let $V = \loc(\alphabar,e^{\alphabar}/C)$, the smallest subvariety of $\G_a^n\times\G_m^n$ defined over $C$ and containing $(\alphabar,e^{\alphabar})$. Since $\delta(\alphabar/C) = 0$ we have $\dim V = n$. By Lemma~\ref{lem:thumbtack}, $F_C$ is of the appropriate form to apply Proposition~\ref{prop:thumbtack}. Thus, replacing $\alphabar$ by $\alphabar/m$ for some $m \in \N^+$ if necessary, we may assume that $V$ is Kummer-generic over $F_C$. There is a matrix $M \in \Mat_{r\cross n}(\Q)$ and a tuple $\gammabar \in C^r$ such that $\abar = M \alphabar + \gammabar$. 

Now let $X'$ be a finite subset of $X$ such that $X \minus X_0 \subs X'$ and $\abar \in \ecl^B(X')$. Let $c_0$ be a finite tuple from $C$ such that $V$ is defined over the field $\Q(c_0)$, $X' \cup \gammabar \subs c_0$, and $c_0 \strong B$. 

Suppose $\qftp_L(\bbar/c_0) = \qftp_L(\abar/c_0)$. Then by Axiom~II there exists $\betabar\in B^n$ such that $\qftp_L(\bbar\betabar/c_0) = \qftp_L(\abar\alphabar)$. In particular, $\betabar$ is $\Q$-linearly independent over $c_0$, and $(\betabar,e^\betabar) \in V$, and $\bbar = M\betabar + \gammabar$.

Since $c_0 \strong B$ we have $\td(\betabar,e^\betabar/c_0,e^{c_0}) - \ldim_\Q(\betabar/c_0) \ge 0$, but $\ldim_\Q(\betabar/c_0) = n$, so $\td(\betabar,e^\betabar/c_0) \ge \td(\betabar,e^\betabar/c_0,e^{c_0}) \ge n$. Since $(\betabar,e^\betabar) \in V$ and $\dim V = n$ we have that $(\betabar,e^\betabar)$ is generic in $V$ over $c_0$. We must show that it is generic in $V$ over $C$, and since $C \strong B$ it is enough to show that $\betabar$ is $\Q$-linearly independent over $C$.

The relation ``$x \in \ecl(\ybar)$'' is given by an \Loo-formula. For any finite tuple $\zbar$, let $C_\zbar = \gen{ \bigcup_{i=1}^n \ecl^B(X_i \cap (X'\cup \zbar))}$. Then $C_\zbar$ is \Loo-definable with parameters $\zbar\cup c_0$, and $C = \bigcup_{\zbar \finsub X_0} C_\zbar$. 

There is an \Loo-formula $\theta(\xbar)$ with parameters $c_0$ expressing ``if $\zbar$ is $\ecl$-independent over $c_0$ then $\xbar$ is $\Q$-linearly independent over $C_\zbar$''. Now $B \models \theta(\alphabar)$ and, by Corollary~\ref{language}, $\alphabar$ and $\betabar$ satisfy the same $\Loo$-formulas with parameters from $c_0$. Hence $B \models \theta(\betabar)$. But then $\betabar$ is \Q-linearly independent over $C$, and hence $(\betabar,e^\betabar)$ is generic in $V$ over $C$. Since $V$ is Kummer-generic over $F_C$, it follows that $F_{C\cup \alphabar}$ and $F_{C \cup \betabar}$ are isomorphic as partial exponential field extensions of $F_C$. Hence by Lemma~\ref{pE-subfields}, $\qftp_L(\bbar/C) = \qftp_L(\abar/C)$, as required.
\end{proof}

\begin{remark}
This proof repairs two holes in the original proof. There was a mistake in the original proof of the ``thumbtack lemma'' in \cite{Zilber06covers}, which was corrected in \cite{BZ11} but only by strengthening the hypotheses to insist that the field $C$ is the union of an independent system in the sense of Lemma~\ref{lem:thumbtack}. We have been unable to prove Lemmas~\ref{field indep} and~\ref{lem:thumbtack} except in the case where the base of the system is of infinite exponential transcendence degree.

Key to our proof was the observation that this assumption does not affect the proof of categoricity. Furthermore, it ensures that all our models are $\Loo$-equivalent, which we use in the proof of Proposition~\ref{excellence} in going from $\bbar$ being linearly independent over $c_0$ to being linearly independent over $C$. This argument was missing from the original proof.
 
The reduction to this ``infinite-dimensional'' case, or more generally to an $\aleph_0$-categorical \Loo-sentence, is one which Shelah made in his work on excellent classes. While in his general setting that reduction is always possible, what was surprising to us is that it appears to be a significant reduction, in that we are still unable to prove the quasiminimal excellence axiom without it.
\end{remark}

Both authors would like to thank the Max Planck Institute for Mathematics, Bonn, where much of this work was completed.

\input{pseudoexp_excellence.bbl}


\end{document}

%% file: macros.tex
\newcommand{\abar}{{\ensuremath{\bar{a}}}}
\newcommand{\bbar}{{\ensuremath{\bar{b}}}}

\newcommand{\xbar}{{\ensuremath{\bar{x}}}}
\newcommand{\ybar}{{\ensuremath{\bar{y}}}}
\newcommand{\zbar}{{\ensuremath{\bar{z}}}}
\newcommand{\mbar}{{\ensuremath{\bar{m}}}}

\DeclareMathOperator{\qftp}{qftp}  

\DeclareMathOperator{\acl}{acl}   

\DeclareMathOperator{\td}{td}  

\DeclareMathOperator{\loc}{Loc}   

\DeclareMathOperator{\ldim}{ldim}  
\DeclareMathOperator{\Mat}{Mat}  

\DeclareMathOperator{\ecl}{ecl} 

\newcommand{\N}{\ensuremath{\mathbb{N}}}

\newcommand{\Q}{\ensuremath{\mathbb{Q}}}

\newcommand{\Cexp}{\ensuremath{\mathbb{C}_{\mathrm{exp}}}}

\newcommand{\Loo}{\ensuremath{L_{\omega_1,\omega}}}

\newcommand{\ga}{\ensuremath{{\mathbb{G}_\mathrm{a}}}}   
\newcommand{\gm}{\ensuremath{{\mathbb{G}_\mathrm{m}}}}  

\renewcommand{\phi}{\varphi}

\renewcommand{\ge}{\ensuremath{\geqslant}}
\newcommand{\tuple}[1]{\ensuremath{\langle #1 \rangle}}
\newcommand{\class}[2]{\ensuremath{\left\{ #1 \,\left|\, #2 \right.\right\}}}

\newcommand{\iso}{\cong}

\newcommand{\subs}{\subseteq} 
\newcommand{\sups}{\supseteq} 

\DeclareMathOperator{\fin}{fin}  
\newcommand{\subsetfin}{\ensuremath{\subseteq_{\fin}}} 
\newcommand{\finsub}{\subsetfin}

\newcommand{\minus}{\ensuremath{\smallsetminus}}

\newcommand{\strong}{\ensuremath{\lhd}} 
\newcommand{\nstrong}{\ensuremath{\not\kern-4pt\lhd\;}} 

\newcommand{\gen}[1]{\ensuremath{\left\langle #1 \right\rangle}} 
\newcommand{\hull}[1]{\ensuremath{\lceil #1\rceil}}

\newcommand{\cross}{\ensuremath{\times}}

\newbox\noforkbox \newdimen\forklinewidth
\forklinewidth=0.3pt \setbox0\hbox{$\textstyle\smile$}
\setbox1\hbox to \wd0{\hfil\vrule width \forklinewidth depth-2pt
  height 10pt \hfil}
\wd1=0 cm \setbox\noforkbox\hbox{\lower 2pt\box1\lower
2pt\box0\relax}
\def\unionstick{\mathop{\copy\noforkbox}\limits}

\def\nonfork_#1{\unionstick_{\textstyle #1}}

\setbox0\hbox{$\textstyle\smile$} \setbox1\hbox to \wd0{\hfil{\sl
/\/}\hfil} \setbox2\hbox to \wd0{\hfil\vrule height 10pt depth
-2pt width
                \forklinewidth\hfil}
\newbox\doesforkbox
\setbox\doesforkbox\hbox{\lower 2pt\box1 \lower
2pt\box2\lower2pt\box0\relax}
\def\nunionstick{\mathop{\copy\doesforkbox}\limits}

\def\fork_#1{\nunionstick_{\textstyle #1}}

\newcommand{\findep}[4]{\ensuremath{#1 \nonfork_{#3}^{#4} #2}}

\newcommand{\leteq}{\mathrel{\mathop:}=}

%% file: pseudoexp_excellence.bbl
\newcommand{\etalchar}[1]{$^{#1}$}